\documentclass[10pt,leqno]{article}
\usepackage{graphicx}
\usepackage{physics}
\usepackage{fancybox}

\usepackage{xcolor}
\usepackage{titlesec}
\usepackage{fancyheadings}
\usepackage{graphicx}
\usepackage{subcaption}
\usepackage{lmodern}
\usepackage{amsmath}
\usepackage{ytableau}
\usepackage{amssymb}
\usepackage{mathrsfs}
\usepackage{subcaption}
\usepackage{graphicx}
\usepackage{hyperref}
\usepackage{youngtab}

\usepackage{xcolor}
\usepackage{amsthm,epsfig,epstopdf,url,array, geometry}
\usepackage{algorithm,algorithmic}
\usepackage[]{tikz}

\makeatletter
\RequirePackage{xkeyval}
\RequirePackage{tikz}
\RequirePackage{amssymb}
\usepackage{indentfirst,csquotes}
\usepackage{physics}
\newcommand{\Exterior}{\mathchoice{{\textstyle\bigwedge}}%
	{{\bigwedge}}%
	{{\textstyle\wedge}}%
	{{\scriptstyle\wedge}}}
\usepackage{amssymb,amsthm,amsmath}
\usepackage{xcolor,paralist,hyperref,titlesec,fancyhdr,etoolbox}
\newtheorem*{remarque}{Remark}
\newtheorem{theorem}{Theorem}[]
\newtheorem{definition}[theorem]{Definition}
\newtheorem{example}[theorem]{Example}

\newtheorem{proposition}[theorem]{Proposition}

\newtheorem{conjecture}[theorem]{Conjecture}

\titleformat{\section}[display]{\normalfont\huge\bfseries\centering}{\centering\chaptertitlename\thechapter}{10pt}{\Large}
\titlespacing*{\section}{0pt}{0ex}{0ex}

\hypersetup{ colorlinks=true, linkcolor=black, filecolor=black, urlcolor=black }

\usepackage{lipsum}

\begin{document}
	\title{KP equation : Sato tau functions and construction of Somos sequence} 
	\author{BENSAID Mohamed\footnote{Université de Lille, France\\
 \hspace*{0.5cm}E-mail adresse : mohamed.bensaid.etu@univ-lille.fr}}

	\maketitle
	
	\let\thefootnote\relax 

	\begin{abstract}
	In this short article, we will reconstruct the KP equation from Plücker relations and provide some generalizations on this topic. Additionally, in the final section, we define the discrete function $\tau$ in a similar manner, leading to the construction of an integer sequence that has not yet been listed in the OEIS. Furthermore, this approach allows us to construct many other sequences that are not listed in the OEIS, This allows us to see a connection between integrable systems and number theory.	\end{abstract}

\section*{Introduction}
The KP equation given by $\partial_x(u_{xxx}+6uu_{x}-4u_{t})+3u_{yy}=0$, is widely studied in physics to describe nonlinear wave motion. It is known that the construction of the KP hierarchy reduces to an equation with an infinite number of unknown functions, but this can be simplified to a single unknown function called the $\tau$ function. This function is unique up to multiplication by a certain constant. You can refer to the standard construction of the KP equation in \cite{1}, \cite{3}, \cite{6}, \cite{9}, \cite{10}, and \cite{13}.
\section*{The standard construction of the KP equation  }
We define the operator (Lax operator)  
\[
L=\partial+\sum_{i\geq 1}u_i\partial^{-i}
\]
where \( u_i = u_i(\mathbf{t}) \). The KP hierarchy equations are all given by:  
\begin{equation}\label{L}
	\partial_{t_{n}}L=[A_n,L], \quad n=1,2,\dots \quad \text{with } A_n=(L^n)_+
\end{equation}  
along with the compatibility condition  
\[
\partial_{m,n}^2L=\partial_{n,m}^2L.
\]

The system \eqref{L} can be rewritten as:  
\begin{equation}\label{l}
	[L,\partial_n-A_n]=0.
\end{equation}

\begin{remarque}  
	In the definition of the operator \( L \), we observe that the coefficient in front of \( \partial^0 \) does not appear. In fact, any element of the form  
	\[
	L = \partial+\sum_{i\geq 0}u_i\partial^{-i}
	\]
	can be rewritten without the term \( u_0 \). Indeed, by setting \( u_0=-\frac{g'}{g} \), the transformed operator \( g^{-1}Lg \) has no \( u_0 \) term.  
	Thus, the class of \( u_0 \) is not relevant to this study.
\end{remarque}  

\begin{proposition}  
	The KP hierarchy equations \eqref{L} are equivalent to:  
	\begin{equation}\label{ZZ}
		\partial_mA_n-\partial_nA_m+[A_n,A_m]=0
	\end{equation}
	and  
	\begin{equation}\label{ee}
		\partial_mA_n^--\partial_nA_m^-+[A_n^-,A_m^-]=0.
	\end{equation}
	As a result, the compatibility condition is satisfied.
\end{proposition}

\begin{example}[Standard construction for the KP equation]  
	Let us take \( n=2, m=3 \). We have  
	\[
	A_2=\partial^2+2u_1, \quad A_3=\partial^3 +3u_1\partial +3\partial u_1+3u_2.
	\]
	Setting \( t_1=x, t_2=y, t_3=t \), expanding  
	\[
	\partial_{t_3}A_2-\partial_{t_2}A_3+[A_2,A_3]=0
	\]
	yields two equations for the unknowns \( u_1, u_2 \). By substituting one into the other, we eliminate \( u_2 \), thus obtaining the KP equation.
\end{example}
\begin{remarque}  
    The condition \eqref{ZZ} is the compatibility condition of the linear system  
    \[
    \frac{\partial \psi}{\partial t_n}=A_n\psi, \quad \text{and} \quad L\psi=z\psi.
    \]
    In fact, if \( L \) is a solution of our KP hierarchy, then the system admits a solution, and there exists a distinguished solution of the system, which we will discuss shortly, called the **Baker function**.
\end{remarque}

The Baker function is given by:  
\[
\psi(\mathbf{t},z)=\mathcal{W}e^{\sum t_kz^k}
\]
where  
\[
\mathcal{W}=1+\sum_{k\geq1}\alpha_k\partial^{-k}.
\]
Setting  
\[
L=W\partial W^{-1}=\partial+\sum_{i\geq 1}u_i\partial^{-i}
\]
and defining  
\[
W^*=1+\sum_{k\geq1}\beta_k\partial^{-k}, \quad \text{with } \beta_k=(-1)^k\alpha_k^{*},
\]
we introduce the adjoint Baker function:  
\[
\psi^*(\mathbf{t},z)=(W^*)^{-1}e^{-\sum t_kz^k}.
\]

\begin{theorem}[Sato \cite{10}]  
    There exists a function \( \tau \) (unique up to multiplication by a certain constant) such that:  
    \begin{equation}
        \psi(\mathbf{t},z)=\frac{\tau(\mathbf{t}-[z^{-1}])}{\tau(\mathbf{t})}e^{\sum_{k\geq1}t_kz^k}
    \end{equation}  
    where  
    \[
    [z]=\left(z, \frac{z^2}{2}, \frac{z^3}{3}, \dots \right).
    \]
\end{theorem}

Now, we reconstruct the KP equation using another method, which leads us to define the discrete tau function, resulting in the construction of Somos sequences.
\section*{Plücker Relation}

Let $V$ be an infinite-dimensional vector space. Consider two subspaces of $V$, $L = \text{span}\{v_{i_1}, \dots\}$ and $L' = \text{span}\{w_{j_1}, \dots\}$ such that $L \cap L' = 0$ and $\operatorname{dim}(V/L \oplus L') = 2$.

We are given two vectors $u_1, u_2 \in V$. 

We define an element $\dots w_{j_2} \wedge w_{j_1} \wedge u_1 \wedge u_2 \wedge v_{i_1} \wedge v_{i_2} \dots \in \Exterior^{\text{max}} V$. 

Once again, this can be written as
$$\bra{L'}\wedge u_1\wedge u_2 \wedge\ket{L},$$
where $\bra{L'} := \dots w_{j_2} \wedge w_{j_1}$ and $\ket{L} := v_{i_1}\wedge v_{i_2} \dots$

To shorten notation we will omit below the sign of the wedge produc

\begin{proposition}[Plücker Relation]
	For all $a, b, c, d \in V$, we have
	$$\langle L'|ab|L\rangle\langle L'|cd|L\rangle-\langle L'|ac|L\rangle\langle L'|bd|L\rangle+\langle L'|ad|L\rangle\langle L'|bc|L\rangle=0 \in (\Exterior^{\text{max}})^{\otimes 2} V.$$
\end{proposition}

\begin{proof}
	Let $R(a, b, c, d)$ denote the right-hand side. We note that $R$ is a multilinear function that vanishes if any of the vectors belong to $L \oplus L'$. Hence, we can consider it as a function of $\mathbb{C}^2 = V/L \oplus L'$, and since $R$ is antisymmetric, the function $R$ is identically zero because $\Exterior^4 \mathbb{C}^2 = 0$.
\end{proof}
\begin{remarque}
	We can even generalize this relation. Indeed, let $L, L'$ be such that $L \cap L' = 0$ and $\operatorname{dim}(V/L \oplus L') = 3$. We can similarly show that for any $a, b, c, x, y, z \in V$, we have:
	\begin{align*}
		\bra{L'} abc \ket{L} \bra{L'} xyz \ket{L} &- \bra{L'} abx \ket{L} \bra{L'} ayz \ket{L} + \\
		&+ \bra{L'} aby \ket{L} \bra{L'} cxz \ket{L} - \bra{L'} abz \ket{L} \bra{L'} axy \ket{L} = 0.
	\end{align*}
\end{remarque}

\section*{Description of the KP Equation}

In this section, we will construct the KP equation again from the $\tau$ function.
\subsection*{Bosons-fermions coresspondance}
We recall many results concerning the correspondence between bosons and fermions.
\begin{definition}[Maya Diagram]
	A Maya diagram is an arrangement of "crosses" and "circles" at each position in $\mathbb{Z} + \frac{1}{2}$ such that all positions $m < 0$ except for a finite number have a cross, and all positions $m > 0$ except for a finite number have a circle. In other words, $M$ is a Maya diagram if $|\{M \cap (\mathbb{Z}_{>0} + \frac{1}{2})\}| < +\infty$ and $|\{(\mathbb{Z}_{<0} + \frac{1}{2}) \cap M^c\}| < +\infty$.
\end{definition}

\begin{proposition}
	$$\mathcal{M}aya \simeq \mathcal{Y}oung \times \mathbb{Z}$$
\end{proposition}
\begin{proof}
To prove this proposition, we present the key idea. Let us take, for example $\lambda=(4,2,2,1)$,
\begin{figure}[h]
    \centering
    \includegraphics[width=10cm, height=5cm]{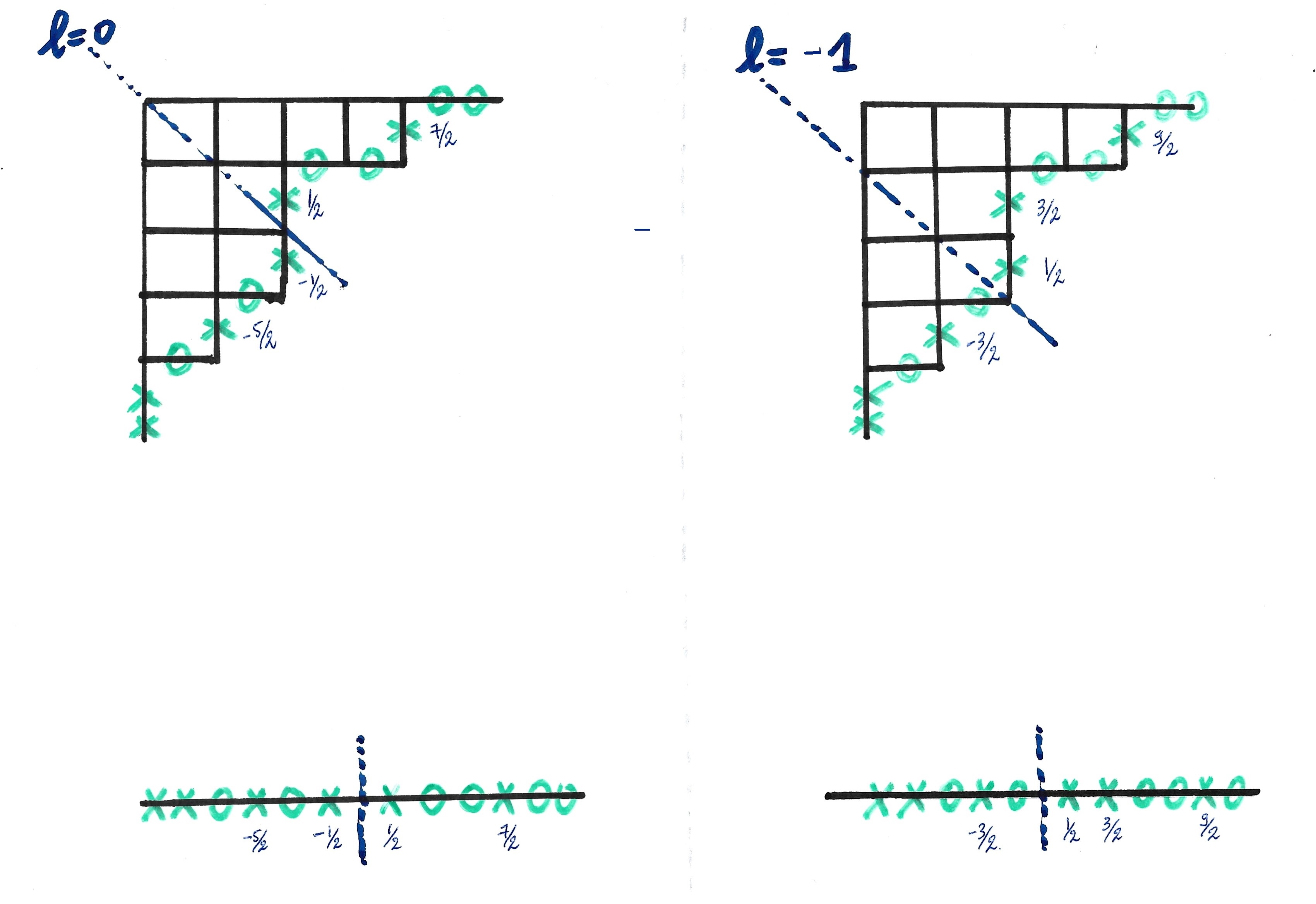}
    \caption{Correspondence between Maya and Young}
\end{figure}

 In this figure, we observe that each pair $(\lambda,l)$ defines the unique Maya diagram, and vice versa.
\end{proof}
The bosonic space $\mathcal{B} = \mathbb{C}[p_1, \ldots; z, z^{-1}] = \bigoplus_{l \in \mathbb{Z}} z^l \mathbb{C}[p_1, \ldots]$ is introduced, along with two operators $p_m = p_m \cdot$ and $p_{-m} = m \frac{\partial}{\partial p_m}$ which act on the space $\mathcal{B}$.

The fermionic space $\mathcal{F} := \bigoplus_{l \in \mathbb{Z}} \mathcal{F}^l := \bigoplus_{l\in \mathbb{Z}}\{v_{i_1} v_{i_2} \cdots \mid i_1 > i_2 > \cdots, i_k = -k + l + \frac{1}{2} \text{ for } k \gg 0\}$, in other words, the basis is indexed by Maya diagrams. Is introduced, with two operators $\{\psi_n, \psi_n^*\}_{n \in \mathbb{Z} + \frac{1}{2}}$ which act on the fermionic space as follows:

$$\psi_n \cdot v_\lambda = v_n v_\lambda$$

$$\psi_n^* \cdot v_\lambda = \begin{cases}
	0 & \text{if } v_n \text{ does not appear as a factor in } v_\lambda, \\
	v_\mu & \text{if } v_\lambda \text{ can be expressed as } v_n v_\mu \text{ for some } v_\mu.
\end{cases}$$

\begin{theorem}[Boson-Fermion Correspondence]
	The bosonic space is isomorphic to the fermionic space, and we have the following correspondences for \(m > 0\):
	\begin{center}
		\begin{tabular}{ccc}
			\(\sum_i \psi_{i+m}\psi_i^*\) & \(\longleftrightarrow\) & \(p_m\) \\ 
			\(\sum_i \psi_{i-m}\psi_i^*\) & \(\longleftrightarrow\) & \(p_{-m} = m\frac{\partial}{\partial p_m}\) \\
			\(\sum_i :\psi_{i}\psi_i^* := \sum_{i<0}\psi_i\psi_i^* - \sum_{i>0}\psi_i^*\psi_i\) & \(\longleftrightarrow\) & \(l\)
		\end{tabular}
	\end{center}
\end{theorem}

\begin{remarque}
	We can also consider the mapping:
	\[
	(n, \lambda) \mapsto \prod_i p_{\lambda_i} |n\rangle
	\]
	where \(\ket{n}\) is defined as
	\[
	\ket{n} := \cdots \times \times \times \times \times \times \bigg|_{n} \circ \circ \circ \cdots
	\]
	
\end{remarque}
\subsection*{KP equation from Plucker} We recall that the $\tau$ function in fermionic image can be written as:
\begin{equation}\label{h}
	\tau(t_1, t_2, \dots) = \bra{0} g \cdot \exp\left(\sum_{k \geq 1} t_k p_k\right) \ket{0} = \bra{g} \exp\left(\sum_{k \geq 1} t_k p_k\right) \ket{0},
\end{equation}
with $p_k = \sum_j \psi_{j+k} \psi^*_{j}$.

\begin{remarque}
	the $g$ is a subspace of the space $V$.
\end{remarque}
Now, let's construct the KP equation from the Plücker relation.

Let $x = t_1$, $y = t_2$, $t = t_3$, and denote:
\begin{align*}
	&\langle g|\exp(xp_1+yp_2+tp_3) = \langle L'| \\
	&v_{-5/2}v_{-7/2}\cdots= |L\rangle\\
	&(a, b, c, d) = (v_{3/2}, v_{1/2}, v_{-1/2}, v_{-3/2})
\end{align*}

\begin{align}
	&\times\times\circ\circ = |0\rangle = v_{-1/2}v_{-3/2}|L\rangle \label{1} \\
	&\times\circ\times\circ = p_1 |0\rangle = v_{1/2}v_{-3/2}|L\rangle \\
	&\circ\times\times\circ = \frac{1}{2}(p_1^2 + p_2)|0\rangle = v_{1/2}v_{-1/2}|L\rangle \\
	&\times\circ\circ\times = \frac{1}{2}(p_1^2 - p_2)|0\rangle = v_{3/2}v_{-3/2}|L\rangle \\
	&\circ\times\circ\times = \frac{1}{3}(p_1^3 - p_3)|0\rangle = v_{3/2}v_{-1/2}|L\rangle \\
	&\circ\circ\times\times = \frac{1}{12}(p_1^4 + 3p_2^2 - 4p_1p_3)|0\rangle = v_{3/2}v_{1/2}|L\rangle \label{6}.
\end{align}

Plücker Relation:
\begin{align*}
	&\langle L'|v_{3/2}v_{1/2}|L\rangle \langle L'|v_{-1/2}v_{-3/2}|L\rangle - \langle L'|v_{3/2}v_{-1/2}|L\rangle \langle L'|v_{1/2}v_{-3/2}|L\rangle + \\
	&+\langle L'|v_{3/2}v_{-3/2}|L\rangle \langle L'|v_{1/2}v_{-1/2}|L\rangle = 0.
\end{align*}

This can be symbolically written as:
$$\circ\circ\times\times \otimes \times\times\circ\circ - \circ\times\circ\times \otimes \times\circ\times\circ + \circ\times\times\circ \otimes \times\circ\circ\times = 0.$$

Substituting \ref{1} -- \ref{6} into the Plücker relation:
\begin{align*}
	&\langle L'| \frac{1}{12}(p_1^4 + 3p_2^2 - 4p_1p_3)|0\rangle \langle L'|L\rangle - \langle L'| \frac{1}{3}(p_1^3 - p_3)|0\rangle \langle L'|p_1|0\rangle |L\rangle + \\
	&+\langle L'| \frac{1}{2}(p_1^2 + p_2)|L\rangle \langle L'| \frac{1}{2}(p_1^2 - p_2)|L\rangle = 0.
\end{align*}

And using that
$$\frac{\partial^{k+l+m} \tau}{\partial x^k \partial y^l \partial z^m} = \langle g | \exp(xp_1 + yp_2 + zp_3) p_1^k p_2^l p_3^m |0\rangle,$$
we obtain:

\begin{align*}
	&(\tau_{xxxx}+3\tau_{yy}-4\tau_{xt})\tau-4(\tau_{xxx}-\tau_t)\tau_x+3(\tau_{xx}+\tau_y)(\tau_{xx}-\tau_y)=0\\
	\Leftrightarrow\quad&(\tau\tau_{xxxx}-4\tau_{xxx}\tau_x+3\tau_{xx}^2)-4(\tau\tau_{xt}-\tau_x\tau_t)+3(\tau\tau_{yy}-\tau_y^2)=0\\
	\textrm{Soit }\tau =e^w\\
	\Leftrightarrow\quad&w_{xxxx}+6w_{xx}^2-4w_{xt}+3w_{yy}=0\\
	\Rightarrow\quad&w_{xxxxx}+6w_{xx}w_{xxx}-4w_{xxt}+3w_{yyx}=0\\
	\Rightarrow\quad&\partial_x(w_{xxxxx}+6w_{xx}w_{xxx}-4w_{xxt})+3w_{yyxx}=0\\
	\textrm{Soit } 	u=w_{xx}\\
	\Leftrightarrow\quad&\partial_x(u_{xxx}+6uu_{x}-4u_{t})+3u_{yy}=0
\end{align*}
\section*{Toda lattice}
Let's begin by defining the function $\tau$ similarly:
\begin{equation}
	\tau(n,x,y) = \bra{g} \exp(xp_{1,1} + yp_{2,1}) \ket{n, -n}
\end{equation}
where $p_{1,1} = \sum \psi_{1,i+1} \psi^*_{1,i}$ and $p_{2,1} = \sum_i \psi_{2,i+1} \psi^*_{2,i}$.

\begin{align*}
	\begin{array}{c}\circ\ \circ\\[-3pt] \times\ \times\end{array} & = \left|\begin{array}{c}n-1\\[-3pt] -n+1\end{array}\right\rangle,\\
	\begin{array}{c}\times\ \times\\[-3pt] \circ\ \circ\end{array} & = \left|\begin{array}{c}n+1\\[-3pt] -n-1\end{array}\right\rangle,\\
	\begin{array}{c}\circ\,\times\\[-3pt]\,\times\,\circ\end{array} & = p_{1,1} \left|\begin{array}{c}n\\[-3pt] -n\end{array}\right\rangle,\\
	\begin{array}{c}\times\,\circ\\[-3pt]\,\circ\,\times\end{array} & = p_{2,1} \left|\begin{array}{c}n\\[-3pt] -n\end{array}\right\rangle,\\
	\begin{array}{c}\circ\,\times\\[-3pt]\,\circ\,\times\end{array} & = p_{2,1} p_{1,1} \left|\begin{array}{c}n\\[-3pt] -n\end{array}\right\rangle,\\
	\begin{array}{c}\times\,\circ\\[-3pt]\,\times\,\circ\end{array} & = \left|\begin{array}{c}n\\[-3pt] -n\end{array}\right\rangle.
\end{align*}

Again, using the same idea from the previous section, the Plücker relation is:
$$\begin{array}{c}\times\ \times\\[-3pt] \circ\ \circ\end{array} \otimes \begin{array}{c}\circ\ \circ\\[-3pt] \times\ \times\end{array} - \begin{array}{c}\circ\,\times\\[-3pt]\,\times\,\circ\end{array} \otimes \begin{array}{c}\times\,\circ\\[-3pt]\,\circ\,\times\end{array} + \begin{array}{c}\times\,\circ\\[-3pt]\,\circ\,\times\end{array} \otimes \begin{array}{c}\circ\,\times\\[-3pt]\,\times\,\circ\end{array} = 0.$$

\begin{align*}
	&\tau(n+1)\tau(n-1) - \tau_x(n) \tau_y(n) + \tau(n) \tau_{xy}(n) = 0,\\
	\Rightarrow\ & \frac{\tau(n+1) \tau(n-1)}{\tau(n)^2} = \frac{\tau_x(n) \tau_y(n)}{\tau(n)^2} - \frac{\tau_{xy}(n)}{\tau(n)} = -\frac{\partial^2}{\partial x \partial y} \log \tau(n),\\
	&\text{Let } v(n) = \ln \frac{\tau(n+1) \tau(n-1)}{\tau(n)^2},\\
	\Rightarrow\ & e^{v(n)} = -\frac{\partial^2}{\partial x \partial y} \ln \tau(n),\\
	\Rightarrow\ & 2e^{v(n)} - e^{v(n+1)} - e^{v(n-1)} = \frac{\partial^2}{\partial x \partial y} (\ln \tau(n+1) + \ln \tau(n-1) - 2 \ln \tau(n)) = v_{xy}(n),\\
	\Rightarrow\ & v_{xy} = 2e^{v(n)} - e^{v(n+1)} - e^{v(n-1)}.
\end{align*}

\textbf{Meaning of the Symbol} $\ket{n,-n}$

$\ket{2,-2} := \begin{array}{c}\times\ \times \times\ \times \times\ \times \times\ \underset{2}{|} \circ\ \circ\ \circ \\[-3pt]\,\times\ \times \times\underset{-2}{|} \circ\ \circ\ \circ\ \circ\ \circ\ \circ\ \circ\end{array} = v_{\frac{3}{2}}v_{\frac{1}{2}}v_{-\frac{1}{2}}v_{-\frac{3}{2}}v_{-\frac{5}{2}}w_{-\frac{5}{2}}v_{-\frac{7}{2}}w_{-\frac{7}{2}}\dots$

\begin{remarque}
	The Plücker relation remains valid; indeed, $V = V_1 \oplus V_2$ and we still choose $\operatorname{dim} V/L' \oplus L = 2$, remember that $\Exterior (L_1\oplus L_2)=\Exterior L_1\bigotimes\Exterior L_2$
\end{remarque}

\section*{The Discrete Function $\tau$}
The main idea behind the construction of the discrete function $\tau$ is similar to that of the generalized Toda, except that in this case, we only consider the integers $n_1, \ldots, n_s$. We define the discrete function $\tau$ by:

\begin{equation}\label{i}
	\tau(n_1, \ldots, n_s) = \bra{g} \ket{n_1, \ldots, n_s}
\end{equation}
with $\operatorname{deg}(\textbf{n})=\sum n_i = 0$.
\begin{theorem}
	The function $\tau$ satisfies the following octahedral relation for all $\textbf{n} \in \mathbb{Z}^s$ such that $\operatorname{deg}(\textbf{n}) = -2$, and for any $\alpha < \beta < \gamma < \delta$, we have
	\begin{equation}\label{s}
		\tau(\textbf{n} + e^{\alpha} + e^{\beta})\tau(\textbf{n} + e^{\gamma} + e^{\delta}) - \tau(\textbf{n} + e^{\alpha} + e^{\gamma})\tau(\textbf{n} + e^{\beta} + e^{\delta}) + \tau(\textbf{n} + e^{\alpha} + e^{\delta})\tau(\textbf{n} + e^{\beta} + e^{\gamma}) = 0
	\end{equation}
	where $e^{\alpha} = (0, \ldots, \underset{\text{in the } \alpha\text{-th position}}{1}, \ldots, 0)$.
\end{theorem}

\begin{proof}
	Note that $\ket{n_{\alpha} + 1} = \psi_{\alpha, n_{\alpha} + \frac{1}{2}} \cdot \ket{n_{\alpha}} = v_{\alpha, n_{\alpha} + \frac{1}{2}} \ket{n_{\alpha}}$.
	
	Plücker relation:
	
	\begin{align*}
		\bra{g} v_{\alpha, n_{\alpha} + \frac{1}{2}} v_{\beta, n_{\beta} + \frac{1}{2}} \ket{\textbf{n}} \cdot \bra{g} v_{\gamma, n_{\gamma} + \frac{1}{2}} v_{\delta, n_{\delta} + \frac{1}{2}} \ket{\textbf{n}} &- \\
		\bra{g} v_{\alpha, n_{\alpha} + \frac{1}{2}} v_{\gamma, n_{\gamma} + \frac{1}{2}} \ket{\textbf{n}} \cdot \bra{g} v_{\beta, n_{\beta} + \frac{1}{2}} v_{\delta, n_{\delta} + \frac{1}{2}} \ket{\textbf{n}} &+ \\
		\bra{g} v_{\alpha, n_{\alpha} + \frac{1}{2}} v_{\delta, n_{\delta} + \frac{1}{2}} \ket{\textbf{n}} \cdot \bra{g} v_{\beta, n_{\beta} + \frac{1}{2}} v_{\gamma, n_{\gamma} + \frac{1}{2}} \ket{\textbf{n}} &= 0
	\end{align*}
	
\end{proof}

\begin{conjecture}
	The inverse is also true.
\end{conjecture}
Let $\sigma \in S_s$be a permutation. Denote by $\sigma(\textbf{n}):=(\sigma(n_1),...,\sigma(n_s))$ the action of
the permutation on $\mathbb{Z}^s$. Define the quadratic form \( q_{\sigma}(\mathbf{n}) \) as follows: $$q_{\sigma}(\textbf{n}):=\sum_{\alpha<\beta, \sigma(\alpha)>\sigma(\beta)}i_{\alpha} i_{\beta}$$
\begin{proposition}
	The permutation group acts on the set of solution octahedral relation $\ref{s}$ by 
	
	$$\tau(\textbf{n})\longrightarrow (-1)^{q_{\sigma}(\textbf{n})}\tau(\sigma(\textbf{n})$$
	
\end{proposition}
\begin{proof}
	It is sufficient to verify that any transposition transforms a solution of $\ref{s}$ to solution of $\ref{s}$.
\end{proof}
\begin{definition}
	A function $\tau$ is said to be of double period if there exists a subgroup $G$ of $A_{s-1} = \{\textbf{n} \in \mathbb{Z}^s \mid \operatorname{deg}(\textbf{n}) = 0\}$ of rank 2 such that $\tau$ is invariant under this subgroup. In particular, $\tau$ can be viewed as a function on the quotient $A_{s-1}/G$.
\end{definition}

\begin{remarque}
	By proposition and the definiton of double period of $\tau$ function taht a subgroup can be encoded by convexes polygons with integer vertices.
\end{remarque}

The double-period function $\tau$ is then defined on the quotient $A_{s-1}/\langle \textbf{a}, \textbf{b} \rangle$. In the following examples, we will explore some constructions of integer sequences from the function $\tau$ by taking $s = 4$.

\begin{example}
	The following polygon:
	\begin{center}
		\begin{tikzpicture}
			\draw[step=.5cm,gray,very thin] (-1.5,-1.5) grid (1.5,1.5);
			\draw[very thick,->] (-1.5,-0.5) -- (1,0);
			\draw[very thick,<-] (-1.5,-0.5) -- (-1,0);
			\draw[very thick,<-] (-1,0) -- (0,0.5);
			\draw[very thick,<-] (0,0.5) -- (1,0);
			\draw (-0.5,0) node {\huge \textcolor{red}{$\cdot$}};
			\draw (0,0) node {\huge \textcolor{red}{$\cdot$}};
			\draw (0.5,0) node {\huge \textcolor{red}{$\cdot$}};
		\end{tikzpicture}
	\end{center}
	with matrix: $\begin{pmatrix}
		5 & -2 & -2 & -1 \\
		1 & 1 & -1 & -1
	\end{pmatrix} = \begin{pmatrix}
		\textbf{a} \\
		\textbf{b}
	\end{pmatrix}$. By performing row operations, the matrix becomes $\begin{pmatrix}
		3 & -4 & 0 & 1 \\
		-4 & 3 & 1 & 0
	\end{pmatrix}$, so the quotient $A_{3}/\langle \textbf{a}, \textbf{b} \rangle$ is isomorphic to the group $\{(l, -l, 0, 0) \mid l \in \mathbb{Z}\}$. Taking $\textbf{n} = (l-1, -l-1, 0, 0)$, the octahedral relation becomes:
	
	\begin{align*}
		\tau(l, -l, 0, 0)\tau(l-1, -l-1, 1, 1) - \\
		\tau(l, -l-1, 1, 0)\tau(l-1, -l, 0, 1) + \\
		\tau(l, -l-1, 0, 1)\tau(l-1, -l, 1, 0) &= 0
	\end{align*}
	
	By the $(\textbf{a}, \textbf{b})$-periodicity of the function $\tau$, we deduce:
	
	\begin{align*}
		\tau(l, -l, 0, 0)\tau(-l, -l, 0, 0) - \\
		\tau(l+4, -l-4, 0, 0)\tau(l-4, -l+4, 0, 0) + \\
		\tau(l-3, -l+3, 0, 0)\tau(l+3, -l-3, 0, 0) &= 0
	\end{align*}
	
	Letting $a_l = \tau(l, -l, 0, 0)$, the previous octahedral relation becomes:
	
	$$a_l^2 - a_{l+4}a_{l-4} + a_{l-3}a_{l+3} = 0$$
	
	$$1, 1, 1, 1, 1, 1, 1, 1, 2, 3, 4, 5, 9, 18, 34, 93, 180, 348, 724, 3033, 9666, 24986, 83761, 261033, \ldots$$
	
	This is sequence A018896 in OEIS, discovered by Somos \cite{11}.
	\end{example}
	\begin{example}
		.\\
		\begin{center}
			\begin{tikzpicture}
				\draw[step=.5cm,gray,very thin] (-1.5,-1.5) grid (1.5,1.5);
				\draw[very thick,->] (-1,-0.5) -- (-0.5,-0.5);
				\draw[very thick,->] (-0.5,-0.5) -- (1,0);
				\draw[very thick,->] (1,0) -- (-0.5,1);
				\draw[very thick,->] (-0.5,1) -- (-1,-0.5);
				\draw (-0.5,0) node {\huge \textcolor{red}{$\cdot$}};
				\draw (0,0) node {\huge \textcolor{red}{$\cdot$}};
				\draw (0.5,0) node {\huge \textcolor{red}{$\cdot$}};
				\draw (0,0.5) node {\huge \textcolor{red}{$\cdot$}};
				\draw (-0.5,0.5) node {\huge \textcolor{red}{$\cdot$}};
			\end{tikzpicture}
		\end{center}
		The corresponding matrix is $\begin{pmatrix}
			1 & 3 & -3 & -1\\
			0 & 1 & 2 & -3
		\end{pmatrix}$. After performing elementary row operations, the matrix becomes $\begin{pmatrix}
			1 & 0 & -9 & 8\\
			0 & 1 & 2 & -3
		\end{pmatrix}$. Thus, the quotient $A_{3}/\langle \textbf{a}, \textbf{b} \rangle$ is isomorphic to the group $\{(0, 0, l, -l) \mid l \in \mathbb{Z}\}$. Taking $\textbf{n} = (-1, -1, l, -l)$, the octahedral relation becomes:
		
		\begin{align*}
			\tau(0, 0, l, -l) \tau(-1, -1, l+1, -l+1) - \\
			\tau(0, -1, l+1, -l) \tau(-1, 0, l, -l+1) + \\
			\tau(0, -1, l, -l+1) \tau(-1, 0, l+1, -l) &= 0
		\end{align*}
		
		By $(\textbf{a}, \textbf{b})$-periodicity of the function $\tau$, we deduce:
		
		\begin{align*}
			\tau(0, 0, l, -l) \tau(0, 0, l-6, -l+6) - \\
			\tau(0, 0, l+3, -l-3) \tau(0, 0, l-9, -l+9) + \\
			\tau(0, 0, l+2, -l-2) \tau(0, 0, l-8, -l+8) &= 0
		\end{align*}
		
		Let $a_l = \tau(0, 0, l, -l)$. The previous recurrence relation becomes:
		
		$$a_l a_{l-6} - a_{l+3} a_{l-9} + a_{l+2} a_{l-8} = 0$$
		
		$$1, 1, 1, 1, 1, 1, 1, 1, 1, 1, 1, 1, 2, 3, 4, 6, 9, 13, 19, 28, 41, 79, 163, 490, 972, 1785, 4270, 9483, \ldots$$
		
		This sequence is not yet listed in OEIS.
	\end{example}

    \begin{example} For $s=5$
	 	.\\
	 	\begin{center}
	 		\begin{tikzpicture}
	 			\draw[step=.5cm,gray,very thin] (-1.5,-1.5) grid (1.5,1.5);
	 			\draw[very thick,->] (0,-0.5) -- (0.5,-0.5);
	 			\draw[very thick,->] (0.5,-0.5) -- (0.5,0);
	 			\draw[very thick,->] (0.5,0) -- (0,0.5);
	 			\draw[very thick,->] (0,0.5) -- (-1,1);
	 			\draw[very thick,->] (-1,1) -- (0,-0.5);
	 			\draw (-0.5,0.5) node {\huge \textcolor{red}{$\cdot$}};
	 			\draw (0,0) node {\huge \textcolor{red}{$\cdot$}};

	 		\end{tikzpicture}
	 	\end{center}
	 	The corresponding matrix is  
\[
\begin{pmatrix}  
1 & 0 & -1 & -2 & 2 \\  
0 & 1 & 1 & 1 & -3  
\end{pmatrix}  
\]
The quotient \( A_4 / \langle \textbf{a}, \textbf{b} \rangle \) is isomorphic to  
\[
\{(0,0, l, s, -l-s) \mid l,s \in \mathbb{Z} \}
\]  
and by taking \( \textbf{n} = (-1,-1,l,s,-l-s) \), we obtain 5 octahedral relations:

	 	 \begin{align*}
	 	 	\tau(0,0,l,s,-l-s)\tau(-1,-1,l+1,s+1,-l-s)-&\\-\tau(0,-1,l+1,s,-l-s)\tau(-1,0,l,s+1,-l-s)+&\\+\tau(0,-1,l,s+1,-l-s)\tau(-1,0,l+1,s,-l-s)&=0
	 	 \end{align*}
	 	 \begin{align*}
	 	 	\tau(0,0,l,s,-l-s)\tau(-1,-1,l+1,s,-l-s+1)-&\\-\tau(0,-1,l+1,s,-l-s)\tau(-1,0,l,s,-l-s+1)+&\\+\tau(0,-1,l,s,-l-s+1)\tau(-1,0,l+1,s,-l-s)&=0
	 	 \end{align*}
	 	 \begin{align*}
	 	 	\tau(0,0,l,s,-l-s)\tau(-1,-1,l,s+1,-l-s+1)-&\\-\tau(0,-1,l,s+1,-l-s)\tau(-1,0,l,s,-l-s+1)+&\\+\tau(0,-1,l,s,-l-s+1)\tau(-1,0,l,s+1,-l-s)&=0
	 	 \end{align*}
	 	 \begin{align*}
	 	 	\tau(0,-1,l+1,s,-l-s)\tau(-1,-1,l,s+1,-l-s+1)-&\\-\tau(0,-1,l,s+1,-l-s)\tau(-1,-1,l+1,s,-l-s+1)+&\\+\tau(0,-1,l,s,-l-s+1)\tau(-1,-1,l+1,s+1,-l-s)&=0
	 	 \end{align*}
	 	 \begin{align*}
	 	 	\tau(-1,0,l+1,s,-l-s)\tau(-1,-1,l,s+1,-l-s+1)-&\\-\tau(-1,0,l,s+1,-l-s)\tau(-1,-1,l+1,s,-l-s+1)+&\\+\tau(-1,0,l,s,-l-s+1)\tau(-1,-1,l+1,s+1,-l-s)&=0
	 	 \end{align*}
	 	 By the (\(\textbf{a},\textbf{b}\))-periodicity of the function \( \tau \) and by setting \( a_{l,s} = \tau(0,0,l,s,-l-s) \), we derive the following recurrence relations:

	 	 \begin{align*}
	 	 	a_{l,s}a_{l+1,s}-a_{l+2,s+1}a_{l-1,s-1}+a_{l+1,s+2}a_{l,s-2}&=0\\
	 	 	a_{l,s}a_{l+1,s-1}-a_{l+2,s+1}a_{l-1,s-2}+a_{l+1,s+1}a_{l,s-2}&=0\\
	 	 	a_{l,s}^2-a_{l+1,s+2}a_{l-1,s-2}+a_{l+1,s+1}a_{l-1,s-1}&=0\\
	 	 	a_{l+2,s+1}a_{l,s}-a_{l+1,s+2}a_{l+1,s}+a_{l+1,s+1}a_{l+1,s}&=0\\
	 	 	a_{l,s-2}a_{l,s}-a_{l-1,s-1}a_{l+1,s-1}+a_{l-1,s-2}a_{l+1,s}&=0
	 	 \end{align*}
	 \end{example}
	 
\begin{remarque}
	To prove that the sequence derived from the octahedral relation is indeed an integer sequence, you can leverage the theory of cluster algebras, particularly using the theorem of Fomin and Zelevinsky.
 \end{remarque}

 \section*{Acknowledgments}
 I would like to express my deep gratitude to Mr. Vladimir Fock, Professor at the University of Strasbourg France.


\begin{thebibliography}{99}
		\addcontentsline{toc}{chapter}{Bibliographie}
		\bibitem{1}{Peter Tingley, Notes Fock space, 2023,  \url{https://arxiv.org/pdf/2211.12463}}
		\bibitem{2}{Ahmed Lesfari, Courbes algébriques complexe et/ou sufaces de Riemann compactes,2023}
		\bibitem{3}{Boris Dubrovin,Integrable Systems and Riemann; \url{https://people.sissa.it/~dubrovin/rsnleq_web.pdf}}
		\bibitem{4}{Macdonald,Symmetric Functions
			and Hall Polynomials,Second edition,1995}
		\bibitem{5}{He-Chi chan,An invitation to q-series,2011}
		\bibitem{6}{M. Sato, Y.Sato, Soliton equations as dynamical systems on infinite dimensional Grass-mann manifolds}
		\bibitem{7}{Victor G.Kac, Infinite dimensional Lie algebra,1995}
		\bibitem{8}{Serge Lang, Complex Analysis}
		\bibitem{9}{Yoko SHIGYO,On Addition Formulae of KP, mKP
			and BKP Hierarchies,2013, https://arxiv.org/pdf/1212.1952}
		\bibitem{10}{L.A.Dickey, Soliton equations and Hamiltonian systems,2003}
		\bibitem{11}{M. Sato, M. Noumi, Soliton equations and the universal Grassmann manifold, 1984}
		\bibitem{12}{John Harnad and  Ferenc Balogh, Tau Functions and their Applications,2021}
		\bibitem{13}{Atsushi Nakayashiki, Sigma Function as A Tau Function,2009,\url{https://arxiv.org/pdf/0904.0846v1}}
        \bibitem{13}{BENSAID Mohamed, Boson-Fermion correspondence and construction of character tables for the symmetric group $S_n$,2024, \url{https://hal.science/hal-04725772v1}}
        
	\end{thebibliography}
\end{document}